\documentclass[11pt,leqno]{amsart}

\usepackage[T1]{fontenc}
\usepackage[utf8]{inputenc}
\usepackage{lmodern}
\usepackage{microtype}

\usepackage{amsmath,amssymb,mathtools}

\usepackage[margin=1.05in]{geometry}
\usepackage{enumitem}
\usepackage[colorlinks=true,linkcolor=blue,citecolor=blue,urlcolor=blue]{hyperref}

\usepackage{aliascnt}
\usepackage[capitalize,noabbrev]{cleveref}

\expandafter\let\csname le\endcsname\leqslant
\expandafter\let\csname ge\endcsname\geqslant
\newcommand{\abs}[1]{\left\lvert #1 \right\rvert}

\newcommand{\1}{\mathbf{1}}
\newcommand{\PP}{\mathbb{P}}
\newcommand{\N}{\mathbb{N}}
\newcommand{\Z}{\mathbb{Z}}

\newcommand{\cA}{\mathcal{A}}
\newcommand{\cS}{\mathcal{S}}

\newtheorem{theorem}{Theorem}[section]
\newaliascnt{lemma}{theorem}
\newtheorem{lemma}[lemma]{Lemma}
\aliascntresetthe{lemma}
\newaliascnt{proposition}{theorem}
\newtheorem{proposition}[proposition]{Proposition}
\aliascntresetthe{proposition}
\newaliascnt{corollary}{theorem}

\aliascntresetthe{corollary}
\theoremstyle{definition}
\newaliascnt{definition}{theorem}
\newtheorem{definition}[definition]{Definition}
\aliascntresetthe{definition}
\theoremstyle{remark}
\newaliascnt{remark}{theorem}

\aliascntresetthe{remark}

\crefname{theorem}{Theorem}{Theorems}
\crefname{lemma}{Lemma}{Lemmas}
\crefname{proposition}{Proposition}{Propositions}
\crefname{corollary}{Corollary}{Corollaries}
\crefname{definition}{Definition}{Definitions}
\crefname{remark}{Remark}{Remarks}

\title[Short intervals for the Romanoff-type sumset]{Short intervals for the Romanoff-type sumset}

\author[Y.~Ding]{Yuchen Ding}
\address{School of Mathematics, Yangzhou University, Yangzhou 225002, People's Republic of China; HUN-REN Alfr\'ed R\'enyi Institute of Mathematics, Budapest, Pf.\ 127, H-1364, Hungary}
\email{ycding@yzu.edu.cn}

\author[J.~Verwee]{Johann Verwee}
\address{Independent Researcher, France}
\email{mverwee@gmail.com}

\date{}

\begin{document}

\begin{abstract}
Let $X$ be large and let $\PP$ denote the set of primes. Fix positive real parameters $r_1,\dots,r_s$ and a parameter $\lambda\geqslant 1$ determined by a balancing relation, and let $\cA_\lambda(X)\subset[1,2X]$ be the associated lacunary set generated by sums of powers of $2$ with polynomially growing exponents. Set $\cS_\lambda:=\PP+\cA_\lambda(X)$.
Fix $\varepsilon>0$, choose $\theta$ with $2/15+\varepsilon<\theta<0.99$, and set $h=X^{\theta}$.
We prove that for all but $O_{\varepsilon}\left(X\exp\left(-c_{\varepsilon}(\log X)^{1/4}\right)\right)$ values of $x\in[X,2X]$, the short interval $(x,x+h]$ contains $\asymp_{\varepsilon} h$ integers of the form $p+a$ with $p\in\PP$ and $a\in\cA_\lambda(X)$.
\end{abstract}

\maketitle

\begin{center}
\small Published in \emph{Journal of Number Theory} \textbf{289} (2026), 114--131.\\
\small DOI: \href{https://doi.org/10.1016/j.jnt.2026.04.002}{10.1016/j.jnt.2026.04.002}.
\end{center}

\section{Introduction}
Let $\PP$ denote the set of prime numbers.
Romanoff proved that the sumset $\PP+\{2^k:k\geqslant 0\}$ has positive lower density~\cite{Romanoff1934}, following earlier questions of de~Polignac~\cite{dePolignac1849a,dePolignac1849b}.
Erd\H{o}s initiated a systematic study of multiplicities and related variants~\cite{Erdos1950} (see also van~der~Corput~\cite{vanDerCorput1950}), and quantitative bounds for Romanoff-type representation functions have been developed more recently, for instance by Chen and Ding~\cite{ChenDing2023}.
See also Elsholtz and Schlage-Puchta~\cite{ElsholtzSchlagePuchta2018} for explicit numerical lower bounds on Romanov's constant.

In this paper we study a short-interval refinement for the structured additive set $\cA_{\lambda}$ introduced by Chen and Xu~\cite{ChenXu2024} (see also Ding and Zhai~\cite{DingZhai}).
For $h=X^{\theta}$ in the range of \Cref{thm:main-elements,thm:main-rep}, we show that for all but a stretched-exponentially small exceptional set of $x\in[X,2X]$ the interval $(x,x+h]$ contains $\asymp h$ elements of the sumset $\PP+\cA_{\lambda}(X)$, and the number of representations in such windows is also $\asymp h$.
In the model case of the Romanoff representation function 
$$ f_{\mathrm{Rom}}(n):=\#\left\{(p,k)\in\PP\times\mathbb N_0:\ p+2^k=n\right\}, $$
we additionally record a short-interval Erd\H{o}s-type statement for a positive proportion of windows, of an integer with unusually large pointwise multiplicity. We now set up the notation and state the main results.

Let $X\geqslant 3$ and set $h=X^{\theta}$ with $\theta\in(0,1)$.
Fix an integer $s\geqslant 2$ and positive real numbers $r_1,\dots,r_s>1$ such that
\begin{equation}\label{eq:reciprocal-sum}
\frac{1}{r_1}+\cdots+\frac{1}{r_{s-1}}<1\leqslant \frac{1}{r_1}+\cdots+\frac{1}{r_s}.
\end{equation}
Let $\lambda\geqslant 1$ be defined by
\begin{equation}\label{eq:lambda-balance}
\frac{1}{r_1}+\cdots+\frac{1}{r_{s-1}}+\frac{1}{\lambda r_s}\ =\ 1,
\end{equation}
Define the additive set $\cA_{\lambda}$ (introduced by Chen--Xu~\cite{ChenXu2024}; see also Ding--Zhai~\cite{DingZhai}) by
\begin{equation}\label{eq:def-A}
\cA_{\lambda}\ :=\ \left\{2^{\lfloor k_1^{r_1}\rfloor}+\cdots+2^{\lfloor k_{s-1}^{r_{s-1}}\rfloor}+2^{\lfloor \lfloor k_s^{\lambda}\rfloor^{r_s}\rfloor}\ :\ k_1,\dots,k_s\in\N\right\}\subset\N.
\end{equation}
For $X\geqslant 3$ we truncate
\[
\cA_{\lambda}(X)\ :=\ \cA_{\lambda}\cap\left[1,2X\right],
\qquad
\cS\ :=\ \PP+\cA_{\lambda}(X).
\]
For real $x\geqslant 1$ we denote by
\[
A_{\lambda}(x)\ :=\ \#\left(\cA_{\lambda}\cap\left[1,x\right]\right)
\]
the associated counting function. The bound $A_{\lambda}(x)\asymp\log x$ holds for all $x\geqslant 3$ (see Chen--Xu~\cite{ChenXu2024}; see also Ding--Zhai~\cite{DingZhai}).
In particular,
\begin{equation}\label{eq:sizeA}
\#\cA_{\lambda}(X)\ \asymp\ \log X.
\end{equation}

In what follows, the parameters $s,r_1,\dots,r_s$ (and hence $\lambda$ and the set $\cA_{\lambda}$) are fixed, and all implied constants are allowed to depend on this data.

Define the representation function
\begin{equation}\label{eq:def-f}
f(n)\ :=\ \sum_{a\in\cA_{\lambda}(X)} \1_{\PP}(n-a),
\end{equation}
so that $f(n)$ counts representations $n=p+a$ with $p\in\PP$ and $a\in\cA_{\lambda}(X)$.
For real $x$ define the window sums
\begin{equation}\label{eq:def-R}
R(x)\ :=\ \sum_{x<n\leqslant x+h} f(n),
\qquad
Q(x)\ :=\ \sum_{x<n\leqslant x+h} f(n)^2,
\end{equation}
and the corresponding count of \emph{elements} of the sumset $\cS$,
\begin{equation}\label{eq:def-Sx}
S(x)\ :=\ \#\left(\cS\cap(x,x+h]\right)\ =\ \sum_{x<n\leqslant x+h}\1_{f(n)\geqslant 1}.
\end{equation}

\medskip

\begin{theorem}[Local density of $\cS$ in short intervals]\label{thm:main-elements}
Fix $\varepsilon>0$. Let $\theta$ satisfy $2/15+\varepsilon<\theta<0.99$, and set $h=X^{\theta}$.
Then for all but $O_{\varepsilon}\left(X\exp\left(-c_{\varepsilon}(\log X)^{1/4}\right)\right)$ integers $x\in[X,2X]$, one has
\[
\#\left(\cS\cap(x,x+h]\right)\ \asymp\ h.
\]
The implied constants depend at most on $\lambda$, $\varepsilon$ and $\theta$.
\end{theorem}

Theorem~\ref{thm:main-elements} is the conclusion we ultimately care about: it asserts that the
sumset $\cS$ meets almost all intervals of length $h$ with the expected order of magnitude.
To access it through Cauchy--Schwarz, it is convenient to also record the corresponding
first-moment statement for the representation function $f$.

\begin{theorem}[Typical number of representations]\label{thm:main-rep}
Fix $\varepsilon>0$. Let $\theta$ satisfy $2/15+\varepsilon<\theta<0.99$, and set $h=X^{\theta}$.
Then for all but $O_{\varepsilon}\left(X\exp\left(-c_{\varepsilon}(\log X)^{1/4}\right)\right)$ integers $x\in[X,2X]$, one has
\[
R(x)\ \asymp\ h.
\]
The implied constants depend at most on $\lambda$, $\varepsilon$ and $\theta$.
\end{theorem}

\medskip

\noindent\emph{Strategy.} The pointwise upper bound $R(x)\ll h$ follows from Brun--Titchmarsh together with $\#\cA_{\lambda}(X)\asymp\log X$.
The almost-all lower bound $R(x)\gg h$ is obtained by restricting to $a\leqslant h/2$ and applying the Guth--Maynard~\cite[Cor.~1.4]{GuthMaynard2024} asymptotic for primes in almost all short intervals (with a stretched-exponential exceptional set).
To pass from representations to elements we combine Cauchy--Schwarz, which relates $R(x)$, $S(x)$ and the local second moment $Q(x)$, with a uniform bound $Q(x)\ll h$.
The latter is proved using Selberg's upper-bound sieve for prime pairs together with Chen--Ding--Xu--Zhai's estimate controlling the contribution of small prime divisors in differences $a_1-a_2$ with $a_1,a_2\in\cA_{\lambda}(X)$.

\medskip

\noindent\emph{Organization of the paper.} Section~2 establishes the first-moment estimate for the local representation count $R(x)$ in almost all short intervals. Section~3 proves a uniform local mean-square bound, leading to the required second-moment control. In Section~4 we combine these inputs with Cauchy--Schwarz to deduce Theorems~\ref{thm:main-elements} and~\ref{thm:main-rep}. Finally, Section~5 proves Theorem~\ref{thm:erdos-short-interval} on large short-interval multiplicities in the Romanoff model.

\section{First moment in short intervals}

We start by exchanging the order of summation in the definition of $R(x)$. From \eqref{eq:def-f} and \eqref{eq:def-R},
\begin{equation}\label{eq:R-as-pi}
R(x)\ =\ \sum_{a\in\cA_{\lambda}(X)} \left( \pi(x+h-a)-\pi(x-a) \right),
\end{equation}
where $\pi$ denotes the prime counting function. To obtain a lower bound for $R(x)$ for almost all $x$, it is convenient to restrict to those $a\leqslant h/2$. Set
\[
\cA_{\lambda}(X;h)\ :=\ \cA_{\lambda}(X)\cap\left[1,\frac{h}{2}\right],
\qquad
A_h\ :=\ \#\cA_{\lambda}(X;h),
\]
so that $A_h\asymp \log h$ by the counting-function estimate $A_{\lambda}(x)\asymp\log x$ (see \cite{ChenXu2024} and \cite{DingZhai}), and hence $A_h\asymp \log X$ in the range $h=X^{\theta}$. Restricting \eqref{eq:R-as-pi} to $\cA_{\lambda}(X;h)$ gives
\begin{equation}\label{eq:R-restrict}
R(x)\ \geqslant\ \sum_{a\in\cA_{\lambda}(X;h)} \left( \pi(x+h-a)-\pi(x-a) \right).
\end{equation}

The lower bound for the prime counts in \eqref{eq:R-restrict} is supplied by Guth--Maynard's almost-all short-interval asymptotic.

\begin{lemma}[Guth--Maynard, almost-all short intervals]\label{lem:GM}
Fix $\varepsilon>0$. Let $X\geqslant 3$ and let $y$ satisfy
\[
 X^{2/15+\varepsilon}\ \leqslant\ y\ \leqslant\ X^{0.99}.
\]
Then for all but $O\left(X\exp\left(-\left(\log X\right)^{1/4}\right)\right)$ integers $t\in[X,2X]$ one has
\[
 \pi(t+y)-\pi(t)\ =\ \frac{y}{\log t}\ +\ O_{\varepsilon}\left(y\exp\left(-\left(\log t\right)^{1/4}\right)\right).
\]
This is \cite[Cor.~1.4]{GuthMaynard2024}, restricted to integer $t$.
\end{lemma}

\medskip

\noindent We now apply \cref{lem:GM} to the shifted intervals in \eqref{eq:R-restrict} and take a union bound over $a\in\cA_{\lambda}(X;h)$.

\begin{proposition}[First moment lower bound]\label{prop:first-moment}
Fix $\varepsilon>0$. Let $\theta$ satisfy $2/15+\varepsilon\leqslant\theta<0.99$ and set $h=X^{\theta}$. Then there exists a set
$\mathcal{E}\subset[X,2X]\cap\Z$ with
\[
\#\mathcal{E}\ \ll_{\varepsilon,\lambda}\ X\exp\left(-c_{\varepsilon}(\log X)^{1/4}\right)
\]
such that
\[
R(x)\ \gg\ h
\]
for all $x\in([X,2X]\cap\Z)\setminus\mathcal{E}$.
\end{proposition}

\begin{proof}
Recall \eqref{eq:R-restrict}. Since $h=X^{\theta}$ with $\theta\in(0,1)$, for $X$ sufficiently large we have $h/2\leqslant X/2$, and hence
\[
\cA_{\lambda}(X;h)=\cA_{\lambda}\cap\left[1,\frac{h}{2}\right]
\qquad\text{and}\qquad
A_h\asymp \log h.
\]
Fix $a\in\cA_{\lambda}(X;h)$ and set $t:=x-a$. For $x\in[X,2X]$ we have $t\in[X/2,2X]$, so $t$ lies in the union of the two dyadic blocks $[X/2,X]$ and $[X,2X]$.

We apply \cref{lem:GM} twice: first with the scale parameter $X' := X/2$ to control $\pi(t+h)-\pi(t)$ for $t\in[X/2,X]$, and then with the scale parameter $X' := X$ to control it for $t\in[X,2X]$. Since $\theta\geqslant 2/15+\varepsilon$, for $X$ sufficiently large we have $(X/2)^{2/15+\varepsilon/2}\leqslant h$. Moreover, since $\theta<0.99$ for $X$ sufficiently large we also have $h=X^{\theta}\leqslant (X/2)^{0.99}$. Consequently the hypotheses of \cref{lem:GM} apply (with $\varepsilon$ replaced by $\varepsilon/2$) on both dyadic blocks. We therefore obtain an exceptional set $\mathcal{E}_a\subset[X,2X]\cap\Z$ with
\[
\#\mathcal{E}_a\ \ll_{\varepsilon}\ X\exp\left(-c_{\varepsilon}(\log X)^{1/4}\right)
\]
such that for all $x\in([X,2X]\cap\Z)\setminus\mathcal{E}_a$ we have
\begin{equation}\label{eq:GM-shifted}
\pi(x+h-a)-\pi(x-a)\ =\ \frac{h}{\log(x-a)}\ +\ O_{\varepsilon}\left(h\exp\left(-c_{\varepsilon}(\log X)^{1/4}\right)\right).
\end{equation}

Let $\mathcal{E}:=\bigcup_{a\in\cA_{\lambda}(X;h)}\mathcal{E}_a$. Using $A_h\ll \log h\ll \log X$, we obtain
\[
\#\mathcal{E}\ \ll_{\varepsilon,\lambda}\ A_h\,X\exp\left(-c_{\varepsilon}(\log X)^{1/4}\right)
\ \ll_{\varepsilon,\lambda}\ X\exp\left(-c'(\varepsilon)(\log X)^{1/4}\right)
\]
for some (slightly smaller) $c'(\varepsilon)>0$.

For $x\in([X,2X]\cap\Z)\setminus\mathcal{E}$, summing \eqref{eq:GM-shifted} over $a\in\cA_{\lambda}(X;h)$ yields
\[
\begin{aligned}
R(x)
&\geqslant\ \sum_{a\in\cA_{\lambda}(X;h)}\left(\pi(x+h-a)-\pi(x-a)\right)
\\
&=\ h\sum_{a\in\cA_{\lambda}(X;h)}\frac{1}{\log(x-a)}\ +\ O_{\varepsilon}\left(A_h h\exp\left(-c_{\varepsilon}(\log X)^{1/4}\right)\right).
\end{aligned}
\]
Since $x-a\asymp X$ uniformly for $x\in[X,2X]$ and $a\leqslant h/2\leqslant X/2$, the main term satisfies
\[
\sum_{a\in\cA_{\lambda}(X;h)}\frac{1}{\log(x-a)}\ \asymp\ \frac{A_h}{\log X}\ \asymp\ \frac{\log h}{\log X},
\]
and the error term is negligible. Since $h=X^{\theta}$, we have $\log h=\theta \log X$, hence $R(x)\gg h$ for all $x\notin\mathcal{E}$.
\end{proof}

\section{Local mean-square bound}

In this section we establish a uniform bound of the correct scale for the local second moment $Q(x)$ defined in \eqref{eq:def-R}. Expanding via the definition \eqref{eq:def-f} of $f$ gives
\begin{equation}\label{eq:Q-expand}
Q(x)\ =\ \sum_{a_1,a_2\in\cA_{\lambda}(X)}\ \sum_{x<n\leqslant x+h} \1_{\PP}(n-a_1)\,\1_{\PP}(n-a_2).
\end{equation}
Changing variables $m=n-a_1$ yields
\begin{equation}\label{eq:Q-shift}
Q(x)\ =\ \sum_{a_1,a_2\in\cA_{\lambda}(X)}\ \sum_{x-a_1<m\leqslant x+h-a_1} \1_{\PP}(m)\,\1_{\PP}\left(m+(a_1-a_2)\right),
\end{equation}
where the inner interval has length exactly $h$.

It is convenient to separate the diagonal and off-diagonal contributions by writing
\[
Q(x)=Q_{\mathrm{diag}}(x)+Q_{\mathrm{off}}(x),
\qquad
Q_{\mathrm{diag}}(x):=\sum_{a\in\cA_{\lambda}(X)}\sum_{x-a<m\leqslant x+h-a}\1_{\PP}(m),
\] 
where $Q_{\mathrm{off}}(x)$ denotes the remaining sum over pairs $(a_1,a_2)\in\cA_{\lambda}(X)^2$ with $a_1\ne a_2$.

\medskip

\noindent We proceed in two steps. We first record a uniform upper-bound sieve estimate for prime pairs in intervals of length $h$, with the dependence on the shift captured by the prime-pair singular series. We then bound the average size of this singular series over differences in $\cA_{\lambda}(X)$.

\subsection*{Prime pairs in short intervals}

We next bound the off-diagonal contribution in \eqref{eq:Q-shift}. The diagonal terms $a_1=a_2$ correspond to counting a single prime in an interval of length $h$ and will be treated separately below. For $a_1\ne a_2$ we are led to sums of the form
\[
\sum_{y<m\leqslant y+h} \1_{\PP}(m)\,\1_{\PP}(m+\Delta),
\]
with length $h$ and a nonzero shift $\Delta$ (in \eqref{eq:Q-shift} one has $y=x-a_1$ and $\Delta=a_1-a_2$). We bound these uniformly in $y$ using Selberg's upper-bound sieve; the expected dependence on $\Delta$ is encoded by the usual prime-pair singular series.

\begin{definition}[Prime-pair singular series]\label{def:S2}
For $\Delta\in\Z$ with $\Delta\ne 0$ define
\[
\mathfrak{S}_2(\Delta)\ :=\
\begin{cases}
2C_2\displaystyle\prod_{\substack{p\mid \Delta\\ p>2}}\frac{p-1}{p-2}, & \Delta\ \text{even},\\[1.2ex]
0, & \Delta\ \text{odd},
\end{cases}
\qquad
C_2\ :=\ \displaystyle\prod_{p>2}\left(1-\frac{1}{(p-1)^2}\right).
\]
\end{definition}

\noindent The factor $\mathfrak{S}_2(\Delta)$ records the local congruence obstructions for the pair of linear forms $m$ and $m+\Delta$ and is the standard weight appearing in upper-bound sieve estimates for prime pairs. We never use $\mathfrak{S}_2(0)$: the case $\Delta=0$ is exactly the diagonal $a_1=a_2$ and will be handled by a one-dimensional prime bound.

\begin{lemma}[Selberg upper-bound sieve for prime pairs in short intervals]\label{lem:pair-short}
Let $h\geqslant 2$ and let $\Delta\in\Z\setminus\{0\}$. Uniformly in real $y$ one has
\[
\sum_{y<m\leqslant y+h} \1_{\PP}(m)\,\1_{\PP}(m+\Delta)\ \ll\ \frac{h}{(\log h)^2}\,\mathfrak{S}_2(\Delta)\ +\ 1,
\]
where $\mathfrak{S}_2$ is as in \cref{def:S2}.
Moreover, uniformly in real $y$,
\[
\sum_{y<m\leqslant y+h} \1_{\PP}(m)\ \ll\ \frac{h}{\log h}.
\]
\end{lemma}

\noindent The first estimate is a standard application of Selberg's upper-bound sieve to the pair of linear forms $m$ and $m+\Delta$ on an interval of length $h$; see Halberstam--Richert~\cite[Ch.~5]{HalberstamRichert}. The implied constant is absolute, and the bound is uniform in $y$ since the local sieve factors depend only on $\Delta$ modulo primes.

The additive term $+1$ accounts for the parity obstruction: if $\Delta$ is odd then $m$ and $m+\Delta$ have opposite parity, so simultaneous primality can occur only when one of the two numbers equals $2$, giving a uniform bound $\leqslant 1$ for every $y$. In our later application, the total contribution of these remainders after summing over $(a_1,a_2)\in\cA_{\lambda}(X)^2$ is $O((\log X)^2)$ by \eqref{eq:sizeA}.

The second estimate follows from Brun--Titchmarsh when $y\geqslant h$. If $y<h$, then every prime counted satisfies $2\leqslant m\leqslant y+h<2h$, so the left-hand side is at most $\pi(2h)\ll h/\log h$; see e.g.~\cite[Ch.~3]{HalberstamRichert}.
\medskip

\noindent Applying \cref{lem:pair-short} to \eqref{eq:Q-shift} with $y=x-a_1$ and $\Delta=a_1-a_2$ bounds each off-diagonal summand by $\frac{h}{(\log h)^2}\,\mathfrak{S}_2(\Delta)+1$. Summing over $a_1\ne a_2$ therefore reduces the off-diagonal contribution to $Q(x)$ to an average of $\mathfrak{S}_2(a_1-a_2)$ over differences in $\cA_{\lambda}(X)$, which we control next.

\subsection*{Averaging the singular series over differences in \texorpdfstring{$\cA_{\lambda}(X)$}{A-lambda(X)}}

To make the sieve bound in \cref{lem:pair-short} effective after summing over $a_1,a_2\in\cA_{\lambda}(X)$, we need to bound the average size of $\mathfrak{S}_2(a_1-a_2)$ for $a_1\ne a_2$. Since $\mathfrak{S}_2(\Delta)$ is essentially a product over primes dividing $\Delta$, large values can only occur when $\Delta$ has many \emph{small} prime factors. For the set $\cA_{\lambda}$, such small prime factors in differences are strongly controlled, via the following estimate.

\begin{lemma}[Chen--Ding--Xu--Zhai, small prime factors in differences]\label{lem:DZ-diff}
Let $x\geqslant 3$ and write $P^{+}(n)$ for the largest prime factor of $n$ (with the convention $P^{+}(1)=1$). Then
\[
\sum_{\substack{a_1,a_2\in\cA_{\lambda}\\ a_1<a_2\leqslant x}}
\ \ \sum_{\substack{d\mid (a_2-a_1)\\ 2\nmid d\\ P^{+}(d)<\log x}}
\frac{\mu^2(d)}{d}
\ \ll_{\lambda}\ (\log x)^2,
\]
as proved in Chen--Xu \cite{ChenXu2024} and Ding--Zhai~\cite[(5)]{DingZhai}.
\end{lemma}

\noindent We now convert \cref{lem:DZ-diff} into a bound for the average singular series. Up to bounded factors coming from primes $p\geqslant \log X$, one has
\[
\mathfrak{S}_2(\Delta)\ \ll\ \prod_{\substack{p\mid \Delta\\ 2<p<\log X}}\left(1+\frac{1}{p}\right),
\]
and the right-hand side expands as a divisor sum supported on odd $d$ with $P^{+}(d)<\log X$.

\begin{lemma}[Average singular series over differences]\label{lem:avg-S2}
For $X\geqslant 3$ one has
\[
\sum_{\substack{a_1,a_2\in\cA_{\lambda}(X)\\ a_1\ne a_2}}
\mathfrak{S}_2(a_1-a_2)\ \ll_{\lambda}\ (\log X)^2.
\]
\end{lemma}

\begin{proof}
Let $\Delta\in\Z$ with $\Delta\ne 0$ and $\abs{\Delta}\leqslant 2X$. Using that
\[
\frac{p-1}{p-2}\ =\ \left(1+\frac{1}{p}\right)\left(1+O\left(\frac{1}{p^2}\right)\right)
\]
uniformly for $p>2$, we obtain
\[
\mathfrak{S}_2(\Delta)\ \ll\ \prod_{\substack{p\mid \Delta\\ p>2}}\left(1+\frac{1}{p}\right).
\]
Since $\abs{\Delta}\leqslant 2X$, there are at most $\omega(\Delta)\leqslant \log(2X)/\log\log X$ prime divisors $p$ of $\Delta$ with $p\geqslant \log X$. For each such $p$ we have $(p-1)/(p-2)=1+O(1/p)=1+O(1/\log X)$, hence
\[
\sum_{\substack{p\mid \Delta\\ p\geqslant \log X}} \frac{1}{p}
\ \leqslant\ \frac{\omega(\Delta)}{\log X}
\ \ll\ \frac{1}{\log\log X}.
\]
Consequently the contribution of primes $p\geqslant \log X$ in the above product is bounded. Therefore
\[
\mathfrak{S}_2(\Delta)\ \ll\ \prod_{\substack{p\mid \Delta\\ 2<p<\log X}}\left(1+\frac{1}{p}\right)
\ =\ \sum_{\substack{d\mid \Delta\\ 2\nmid d\\ P^{+}(d)<\log X}}\frac{\mu^2(d)}{d}.
\]
Summing this bound over pairs $(a_1,a_2)\in\cA_{\lambda}(X)^2$ and noting that $a_1,a_2\leqslant 2X$ on this range, we may apply \cref{lem:DZ-diff} with $x=2X$ to get
\[
\sum_{\substack{a_1,a_2\in\cA_{\lambda}(X)\\ a_1<a_2}}
\mathfrak{S}_2(a_1-a_2)\ \ll_{\lambda}\ (\log X)^2.
\]
The stated bound follows by symmetry.
\end{proof}

\noindent We now combine the diagonal estimate (a single prime in a length-$h$ interval) with the off-diagonal estimate (prime pairs with shift $\Delta=a_1-a_2$ averaged over $\cA_{\lambda}(X)$) to obtain a uniform bound for $Q(x)$. This is the sole $L^2$-input needed in the Cauchy--Schwarz step.

\begin{proposition}[Uniform local mean-square bound]\label{prop:local-mean-square}
Let $h=X^{\theta}$ with $\theta\in(0,1)$. Then for every real $x\in[X,2X]$ one has
\[
Q(x)\ \ll\ h.
\]
The implied constant depends at most on $\lambda$ and $\theta$.
\end{proposition}

\begin{proof}
Split \eqref{eq:Q-shift} into the diagonal contribution $a_1=a_2$ and the off-diagonal contribution $a_1\ne a_2$.
For the diagonal part we use the second bound in \cref{lem:pair-short} to obtain
\[
\sum_{a\in\cA_{\lambda}(X)}\ \sum_{x-a<m\leqslant x+h-a} \1_{\PP}(m)
\ \ll\ \#\cA_{\lambda}(X)\,\frac{h}{\log h}
\ \asymp\ \frac{\log X}{\log h}\,h
\ \asymp\ h,
\]
since $h=X^{\theta}$.
For the off-diagonal part we apply \cref{lem:pair-short} with $\Delta=a_1-a_2\ne 0$, sum over $(a_1,a_2)$ to get
\[
\sum_{\substack{a_1,a_2\in\cA_{\lambda}(X)\\ a_1\ne a_2}}
\ \sum_{x-a_1<m\leqslant x+h-a_1} \1_{\PP}(m)\,\1_{\PP}(m+a_1-a_2)
\ \ll\ \frac{h}{(\log h)^2}\sum_{\substack{a_1,a_2\in\cA_{\lambda}(X)\\ a_1\ne a_2}}\mathfrak{S}_2(a_1-a_2)\ +\ \#\cA_{\lambda}(X)^2.
\]
The first term on the right is $\ll h$ by \cref{lem:avg-S2} and the identity $\log h=\theta \log X$. The remaining additive term $\#\cA_{\lambda}(X)^2$ is polylogarithmic in $X$ by \eqref{eq:sizeA}, hence it is absorbed into the $\ll h$ bound. This gives $Q(x)\ll h$ uniformly in $x$.
\end{proof}

\section{Almost-all conclusions}

We now pass from information on the \emph{number of representations} in a window to information on the \emph{number of distinct integers} represented in that window.
This is the standard Cauchy--Schwarz mechanism behind Romanoff-type positive-density results: a large first moment for the representation function, together with a controlled second moment, forces many distinct sums.

Recall the definitions of $R(x)$ and $Q(x)$ from \eqref{eq:def-R}, and of $S(x)$ from \eqref{eq:def-Sx}. Here $R(x)$ counts representations in $(x,x+h]$, $Q(x)$ measures collisions between representations, and $S(x)$ counts elements of the sumset in the window.
Cauchy--Schwarz gives
\begin{equation}\label{eq:CS}
R(x)^2\ \leqslant\ S(x)\,Q(x),
\qquad\text{hence}\qquad
S(x)\ \geqslant\ \frac{R(x)^2}{Q(x)}.
\end{equation}

We will repeatedly use \eqref{eq:CS} in the following quantitative form.
If, for some $\alpha,\beta>0$, one has
\[
  R(x)\ \geqslant\ \alpha h
  \qquad\text{and}\qquad
  Q(x)\ \leqslant\ \beta h,
\]
then necessarily
\[
  S(x)\ \geqslant\ \frac{\alpha^2}{\beta}\,h.
\]

We combine the first-moment lower bound for $R(x)$ with the uniform local mean-square bound for $Q(x)$ to deduce \Cref{thm:main-elements}.
Theorem~\ref{thm:main-rep} is then a direct consequence of the same lower bound together with the pointwise upper bound $R(x)\ll h$ from the diagonal estimate in Proposition~\ref{prop:local-mean-square}.

\begin{proof}[Proof of \cref{thm:main-elements}]
By \cref{prop:first-moment}, there exists a subset $\mathcal{E}\subset[X,2X]\cap\Z$ with $\#\mathcal{E}\ll_{\varepsilon,\lambda} X\exp\left(-c_{\varepsilon}(\log X)^{1/4}\right)$ such that $R(x)\gg h$ for all $x\in\left([X,2X]\cap\Z\right)\setminus\mathcal{E}$. On the other hand, \cref{prop:local-mean-square} gives the uniform bound $Q(x)\ll h$ for every $x\in[X,2X]$. Inserting these estimates into \eqref{eq:CS} yields
\[
S(x)\ \geqslant\ \frac{R(x)^2}{Q(x)}\ \gg\ h
\]
for all $x\in\left([X,2X]\cap\Z\right)\setminus\mathcal{E}$. Since $h=X^{\theta}\to\infty$, we have $h\geqslant 1$ for $X$ large, and the trivial bound $S(x)\leqslant h+1\leqslant 2h$ completes the proof.
\end{proof}

\begin{proof}[Proof of \cref{thm:main-rep}]
The lower bound $R(x)\gg h$ for all but $O_{\varepsilon}\left(X\exp\left(-c_{\varepsilon}(\log X)^{1/4}\right)\right)$ integers $x\in[X,2X]$ is \cref{prop:first-moment}. The pointwise upper bound $R(x)\ll h$ holds for every $x\in[X,2X]$ by the diagonal estimate in the proof of Proposition~\ref{prop:local-mean-square} (equivalently, by the second estimate in Lemma~\ref{lem:pair-short} and $\#\cA_{\lambda}(X)\asymp\log X$). Combining these bounds yields $R(x)\asymp h$ for all but $O_{\varepsilon}\left(X\exp\left(-c_{\varepsilon}(\log X)^{1/4}\right)\right)$ integers $x\in[X,2X]$.
\end{proof}

\section{A representation-function direction}
\label{sec:romanoff-model}

Theorems~\ref{thm:main-elements} and~\ref{thm:main-rep} provide almost-all information on the support of $f$ (through the sumset $\cS$) and on the total number of representations in a typical window $(x,x+h]$. A natural next question is to understand pointwise multiplicities inside a typical window, for instance higher local moments $\sum_{x<n\leqslant x+h} f(n)^k$ or the frequency of large values of $f(n)$. We do not pursue such refinements here, but we record an illustrative short-interval variant of a classical theorem of Erd\H{o}s for a model Romanoff-type representation function.
Such problems fit into the Romanoff viewpoint discussed in the introduction; see also de~Polignac~\cite{dePolignac1849a,dePolignac1849b} and Romanoff~\cite{Romanoff1934}.

Throughout this section we write $\log_1 x:=\log x$ and $\log_{j+1}x:=\log\left(\log_j x\right)$ for $j\geqslant 1$ (for $x$ sufficiently large).

To keep a concrete model in mind, let
\[
  f_{\mathrm{Rom}}(n)\ :=\ \#\left\{(k,p)\in\N\times\PP : n=2^k+p\right\}.
\]
Erd\H{o}s proved that $\limsup_{n\to\infty} f_{\mathrm{Rom}}(n)/\log_2 n>0$; see~\cite{Erdos1950}. 
The next theorem gives two short-interval refinements of this phenomenon: one always finds a large value within a window of length $\asymp X\exp(-c_1\sqrt{\log X})$, and, more generally, for any fixed $0<\theta<1$ a positive proportion of windows of length $h=X^{\theta}$ contain an integer with multiplicity $\gg \log_2 X$.

\begin{theorem}[Large multiplicities in short intervals]\label{thm:erdos-short-interval}~
\begin{enumerate}[label=(\roman*),leftmargin=*]
\item There exist constants $c_1,c_2>0$ such that for $X$ sufficiently large there exists\\$n\in\left(X,X+X\exp(-c_1\sqrt{\log X})\right]$ with
\[
  f_{\mathrm{Rom}}(n)\ \geqslant\ c_2\,\log_2 n.
\]
\item Let $0<\theta<1$ and set $h:=X^{\theta}$. There exist constants $c_1(\theta),c_2(\theta)>0$ such that, for $X$ sufficiently large, for at least $c_1(\theta)\,X$ integers $x\in(X,2X]$ there exists $n\in(x,x+h]$ with
\[
  f_{\mathrm{Rom}}(n)\ \geqslant\ c_2(\theta)\,\log_2 n.
\]
\end{enumerate}
\end{theorem}

The order of magnitude $\log_2 X$ is dictated by the ratio $d/\varphi(d)$ for a modulus $d$ of size $\exp\left(c\sqrt{\log X}\right)$, since then $d/\varphi(d)\asymp \log\log d\asymp \log_2 X$ by Mertens' theorem.
The trivial pointwise bound
\[
  f_{\mathrm{Rom}}(n)\ \leqslant\ \#\{k:2^k\leqslant n\}\ \leqslant\ \lfloor\log n/\log 2\rfloor
\]
shows that $\log_2 X$ is still far from the largest possible value for exceptional integers.

We will use the following classical prime number theorem in arithmetic progressions with an exceptional modulus.

\begin{lemma}\label{lem:prachar}
There exist absolute constants $c_{3}, c_{4}>0$ satisfying the following property: for every sufficiently large $x$ there exists an integer $k^{\ast}$ with $1<k^{\ast}\leqslant \exp\left(c_{3}\sqrt{\log x}\right)$ (depending on $x$)
 such that for every modulus $k\leqslant \exp\left(c_{3}\sqrt{\log x}\right)$ with $k^{\ast}\nmid k$, and every residue class $\ell\ (\mathrm{mod}\ k)$ with $(\ell,k)=1$, one has
\[
  \pi(x;k,\ell)\ =\ \frac{1}{\varphi(k)}\int_{2}^{x}\frac{1}{\log t}\,\mathrm{d}t+O\left(x\exp\big(-c_4\sqrt{\log x}\big)\right).
\]
\end{lemma}
\begin{proof}
This is a standard consequence of the Landau--Page theorem together with the explicit formula for primes in arithmetic progressions.
For completeness, we briefly indicate the main points, referring to \cite[Ch.~18]{IwaniecKowalski} for full details.

Choose an absolute constant $c_{3}>0$ sufficiently small and set
\[
  Q\ :=\ \exp\left(c_{3}\sqrt{\log x}\right).
\]
By the Landau--Page theorem, for this range of moduli there is at most one primitive real Dirichlet character $\chi^{\ast}$ of conductor $q^{\ast}\leqslant Q$ whose $L$-function has a real zero extremely close to $1$ (a possible Siegel zero) in the usual zero-free region up to height $T$.
If such a character exists, we set $k^{\ast}:=q^{\ast}$; otherwise we set $k^{\ast}:=2$.
For every modulus $k\leqslant Q$ with $k^{\ast}\nmid k$, all Dirichlet $L$-functions modulo $k$ satisfy a zero-free region up to height $T=\exp\left(c_{0}\sqrt{\log x}\right)$ for some absolute $c_{0}>0$.
More precisely, for any Dirichlet character $\chi\ (\mathrm{mod}\ k)$ one has
\[
  L(s,\chi)\neq 0 \qquad \left(\Re(s)\geqslant 1-\frac{c}{\log\left(kT\right)},\ |\Im(s)|\leqslant T\right)
\]
for some absolute constant $c>0$.
In the truncated explicit formula for $\psi(x;k,\ell)$, each zero $\rho=\beta+i\gamma$ with $|\gamma|\leqslant T$ contributes $\ll x^{\beta}/|\rho|$, and the above region gives
\[
  x^{\beta} \leqslant x\exp\left(-\frac{c\log x}{\log\left(kT\right)}\right)
  \leqslant x\exp\left(-c'\sqrt{\log x}\right),
\]
since $\log\left(kT\right)\ll \sqrt{\log x}$ for $k\leqslant Q$.
The remaining truncation error is $\ll x\log^{2} x/T \ll x\exp\left(-c_{0}\sqrt{\log x}\right)\log^{2} x$.
Collecting these contributions yields
\[
  \psi(x;k,\ell)\ =\ \frac{x}{\varphi(k)}+O\left(x\exp\left(-c_{4}\sqrt{\log x}\right)\right),
\]
uniformly for $(\ell,k)=1$, for some absolute $c_{4}>0$.
A partial summation then gives the stated estimate for $\pi(x;k,\ell)$.
\end{proof}

We now turn to the proof of \cref{thm:erdos-short-interval}. The key input is \cref{lem:prachar}, applied at two nearby points to obtain prime-counting asymptotics in the relevant residue classes.

\begin{proof}[Proof of \cref{thm:erdos-short-interval}]
Let $X$ be sufficiently large.

\smallskip
\noindent\emph{Proof of (i).}
Let $k_{1}^{\ast}$ and $k_{2}^{\ast}$ be the exceptional moduli associated to $X$ and $X+\frac{1}{2}X\exp(-c_1\sqrt{\log X})$ in \Cref{lem:prachar}, respectively.
Let $p_{i}^{\ast}:=P^{+}(k_{i}^{\ast})$ denote their largest prime factors.
Define
\[
  d\ :=\ \prod_{\substack{3\leqslant p\leqslant \frac{c_{3}}{2}\sqrt{\log X}\\ p\neq p_{1}^{\ast},\ p\neq p_{2}^{\ast}} } p.
\]
Then by construction neither $k_{1}^{\ast}$ nor $k_{2}^{\ast}$ divides $d$.
By the prime number theorem and Mertens' theorem we have
\[
  d\ =\ \exp\left(\left(\frac{c_{3}}{2}+o(1)\right)\sqrt{\log X}\right) \leqslant \exp\left(c_{3}\sqrt{\log X}\right),
  \qquad
  \frac{d}{\varphi(d)}\ \asymp\ \log_2 X.
\]

Consider
\[
  \Sigma\ :=\ \sum_{\substack{X<n\leqslant X+X\exp(-c_1\sqrt{\log X})\\ d\mid n}} f_{\mathrm{Rom}}(n).
\]
Expanding $f_{\mathrm{Rom}}$ and restricting to $2^k\leqslant \frac{1}{2}X\exp(-c_1\sqrt{\log X})$ yields
\[
\begin{aligned}
  \Sigma
  &=\sum_{\substack{X<n\leqslant X+X\exp(-c_1\sqrt{\log X})\\ d\mid n}}\ \sum_{\substack{2^k+p=n}} 1
  =\sum_{k\geqslant 0}\ \sum_{\substack{X<2^k+p\leqslant X+X\exp(-c_1\sqrt{\log X})\\ p\equiv -2^k\ (\mathrm{mod}\ d)}} 1\\
  &\geqslant
  \sum_{2^k\leqslant \frac{1}{2}X\exp(-c_1\sqrt{\log X})}\ \sum_{\substack{X<p\leqslant X+\frac{1}{2}X\exp(-c_1\sqrt{\log X})\\ p\equiv -2^k\ (\mathrm{mod}\ d)}} 1.
\end{aligned}
\]
Set $H:=X\exp\left(-c_1\sqrt{\log X}\right)/2$. Since $d\leqslant \exp\left(c_3\sqrt{\log X}\right)$, applying \Cref{lem:prachar} at $x=X$ and at $x=X+H$ yields, for each residue class $a\pmod d$ with $(a,d)=1$,
\begin{align*}
\#\left\{X<p\leqslant X+H:\ p\equiv a\pmod d\right\}
&\geqslant \frac{1}{\varphi(d)}\int_X^{X+H}\frac{\mathrm{d}t}{\log t}
+O\left(X\exp\left(-c_4\sqrt{\log X}\right)\right)\\
&\geqslant \frac{H}{3\varphi(d)\log X}
+O\left(X\exp\left(-c_4\sqrt{\log X}\right)\right)\\
&\geqslant \frac{H}{6\varphi(d)\log X},
\end{align*}
provided that $c_1+c_3\leqslant c_4$. Since the number of integers $k\geqslant 0$ with $2^k\leqslant \frac{1}{2}X\exp(-c_1\sqrt{\log X})$ is at least $(\log X)/(2\log 2)$, we obtain
\[
  \Sigma\ \geqslant\ \frac{1}{24\log 2}\frac{X}{\varphi(d)}\exp(-c_1\sqrt{\log X}).
\]

On the other hand, the interval $\left(X,X+X\exp(-c_1\sqrt{\log X})\right]$ contains at most 
$$
\frac{X\exp(-c_1\sqrt{\log X})}{d}+1
$$ 
multiples of $d$, hence there exists an $n\in\left(X,X+X\exp(-c_1\sqrt{\log X})\right]$ with $d\mid n$ such that
\[
  f_{\mathrm{Rom}}(n)\ \geqslant\ \frac{\Sigma}{X\exp(-c_1\sqrt{\log X})/d+1}\ \gg\ \frac{d}{\varphi(d)}\ \gg\ \log_2 X\ \asymp\ \log_2 n.
\]
This proves (i).

\smallskip
\noindent\emph{Proof of (ii).}
Let $0<\theta<1$ and set $h:=X^{\theta}$.
Let $k_{1}^{\ast}$ and $k_{2}^{\ast}$ be the exceptional moduli associated to $X+h/2$ and $2X$ in \Cref{lem:prachar}, respectively, and let $p_{i}^{\ast}:=P^{+}(k_{i}^{\ast})$ denote their largest prime factors.
With the modulus
\[
  d\ :=\ \prod_{\substack{3\leqslant p\leqslant \frac{c_{3}}{2}\sqrt{\log X}\\ p\neq p_{1}^{\ast},\ p\neq p_{2}^{\ast}} } p,
\]
we have $d\leqslant \exp\left(c_{3}\sqrt{\log X}\right)$ and again $d/\varphi(d)\asymp \log_2 X$.

Consider
\[
  \Sigma\ :=\ \sum_{X<x\leqslant 2X}\ \sum_{\substack{x<n\leqslant x+h\\ d\mid n}} f_{\mathrm{Rom}}(n).
\]
Expanding $f_{\mathrm{Rom}}$ and restricting to $2^k\leqslant h/2$ yields
\[
\begin{aligned}
  \Sigma
  &=\sum_{X<x\leqslant 2X}\ \sum_{\substack{x<n\leqslant x+h\\ d\mid n}}\ \sum_{\substack{2^k+p=n}} 1\\
  &=\sum_{X<x\leqslant 2X}\ \sum_{\substack{x<2^k+p\leqslant x+h\\ p\equiv -2^k\ (\mathrm{mod}\ d)}} 1 \\
  &\geqslant \sum_{2^k\leqslant h/2}\ \sum_{\substack{X<x\leqslant 2X\\ x<2^k+p\leqslant x+h\\ p\equiv -2^k\ (\mathrm{mod}\ d)}} 1.
\end{aligned}
\]
If $2^k\leqslant h/2$ and $x<p\leqslant x+h/2$, then $x<2^k+p\leqslant x+h$. Therefore
\[
  \Sigma
  \geqslant \sum_{2^k\leqslant h/2}\ \sum_{\substack{X<p\leqslant 2X+h/2\\ p\equiv -2^k\ (\mathrm{mod}\ d)}}\ \sum_{\substack{p-h/2<x\leqslant p\\ X<x\leqslant 2X}} 1
  \geqslant \frac{h}{3}\sum_{2^k\leqslant h/2}\ \sum_{\substack{X+h/2<p\leqslant 2X\\ p\equiv -2^k\ (\mathrm{mod}\ d)}} 1
\]
for $X$ sufficiently large. Applying \Cref{lem:prachar} at $x=X+h/2$ and at $x=2X$ yields, for each residue class $a\ (\mathrm{mod}\ d)$ coprime to $d$,
\[
  \#\left\{X+h/2<p\leqslant 2X : p\equiv a\ (\mathrm{mod}\ d)\right\}\ \geqslant\ \frac{X}{2\varphi(d)\log X}
\]
once $X$ is large enough. Using $\#\{k\in\N:2^k\leqslant h/2\}\geqslant (\log h)/(3\log 2)$, we obtain
\[
  \Sigma\ \geqslant\ C_0\,h\,\frac{\log h}{\log X}\,\frac{X}{\varphi(d)}
  \ =\ C_0\,\theta\,\frac{hX}{\varphi(d)}
\]
for some absolute constant $C_0>0$.

On the other hand, for every $x\in(X,2X]$ set
\[
  S_d(x)\ :=\ \sum_{\substack{x<n\leqslant x+h\\ d\mid n}} f_{\mathrm{Rom}}(n).
\]
Expanding $f_{\mathrm{Rom}}$ and using the condition $d\mid n$ gives
\[
  S_d(x)
  \ =\ 
  \sum_{k\geqslant 0}\ \#\left\{x-2^k<p\leqslant x+h-2^k : p\in\PP,\ p\equiv -2^k\ (\mathrm{mod}\ d)\right\}.
\]
The summand is empty unless $2^k\leqslant 2X+h$, hence the sum runs over $\ll \log X$ values of $k$. Since $d\leqslant \exp\left(c_{3}\sqrt{\log X}\right)$ and $h=X^{\theta}$, we have $h\geqslant d$ and
\[
  \log\left(\frac{h}{d}\right)\ =\ \theta\log X+O\left(\sqrt{\log X}\right)
  \ \geqslant\ \frac{\theta}{2}\log X
\]
for $X$ sufficiently large. By the Brun--Titchmarsh inequality in arithmetic progressions (see, e.g.,~\cite[Chapter~18]{IwaniecKowalski}), each term is
\[
  \ll \frac{h}{\varphi(d)\log(h/d)}
  \ \ll_{\theta}\ \frac{h}{\varphi(d)\log X},
\]
uniformly in $x$ and $k$. Summing over $k$ yields the uniform bound
\[
  S_d(x)\ \leqslant\ C_1(\theta)\,\frac{h}{\varphi(d)}
\]
for some constant $C_1(\theta)>0$ and all $x\in(X,2X]$ (once $X$ is large enough).

Set
\[
  t\ :=\ \frac{C_0\theta}{4}\,\frac{h}{\varphi(d)}.
\]
Let $\mathcal{X}$ be the set of integers $x\in(X,2X]$ for which $S_d(x)\geqslant t$. Then, using the uniform upper bound above,
\[
  \Sigma
  \leqslant X\,t + \#\mathcal{X}\left(C_1(\theta)\frac{h}{\varphi(d)}-t\right).
\]
Since $\Sigma\geqslant C_0\theta\,hX/\varphi(d)=4Xt$, we deduce that
\[
  \#\mathcal{X}\ \geqslant\ \frac{3Xt}{C_1(\theta)h/\varphi(d)}\ \gg_{\theta}\ X.
\]
Fix $x\in\mathcal{X}$. Since $(x,x+h]$ contains at most $h/d+1\leqslant 2h/d$ multiples of $d$, there exists $n\in(x,x+h]$ such that $d\mid n$ and
\[
  f_{\mathrm{Rom}}(n)\ \geqslant\ \frac{t}{2h/d}\ \gg_{\theta}\ \frac{d}{\varphi(d)}\ \gg\ \log_2 X\ \asymp\ \log_2 n.
\]
This completes the proof of (ii).
\end{proof}
\subsection*{A single-window question}

Part (i) of Theorem \ref{thm:erdos-short-interval} extends slightly an old result of Erd\H{o}s \cite[Theorem 1]{Erdos1950}, which was motivated by a problem of P\'al Tur\'an, communicated privately to Erd\H{o}s, asking whether the function $f_{\mathrm{Rom}}(n)$ is unbounded. Erd\H{o}s' result
$\limsup_{n\to\infty} f_{\mathrm{Rom}}(n)/\log_2 n>0$
was recently extended in a different manner to general sequences in \cite{ChenDing,ChenDing2023}. Based on Theorem \ref{thm:erdos-short-interval},
it is natural to ask for a single-window analogue of \Cref{thm:erdos-short-interval}.

{\it Tur\'an's problem revisited.}
Fix an integer $A\geqslant 1$. Does there exist $\theta\in(0,1)$ such that, for all sufficiently large $X$, one can find an integer $n\in(X,X+h]$ with $f_{\mathrm{Rom}}(n)\geqslant A$?

For $A=1$ the answer is positive for every $\theta>0.52$, by a theorem of Li on primes in short intervals~\cite{LiShortIntervals}.
This sharpens the classical exponent $0.525$ of Baker, Harman and Pintz~\cite{BHP2001}.
Indeed, for all sufficiently large $y$ the interval $[y-y^{0.52},y]$ contains a prime.
Taking $y:=X+h-1$ and assuming $\theta>0.52$, we have $y^{0.52}\leqslant h/2$ for $X$ sufficiently large, hence $[y-y^{0.52},y]\subset[X,X+h-1]$.
Thus there exists a prime $p\in[X,X+h-1]$, and then $n:=p+1\in(X,X+h]$ with $f_{\mathrm{Rom}}(n)\geqslant 1$.

For $A\geqslant 2$ we do not know how to obtain such a uniform-in-$X$ statement.
Informally, one would need to force, inside a prescribed interval of length $h$, several simultaneous primality conditions of the shape $n-2^k\in\mathbb P$, which is closer in spirit to prime-constellation problems in short intervals than to the averaged Romanoff method.

We note, however, that two weaker statements follow from existing results.
First, \Cref{thm:erdos-short-interval} already implies that for every fixed $A$ and every fixed $\theta\in(0,1)$, a positive proportion of $x\in(X,2X]$ have the property that $(x,x+h]$ contains some $n$ with $f_{\mathrm{Rom}}(n)\geqslant A$.
Second, by Maynard's theorem on primes in admissible tuples~\cite{Maynard2015}, for every $A$ there exist infinitely many $n$ with $f_{\mathrm{Rom}}(n)\geqslant A$.
For instance, choose an integer $r\geqslant 1$ and pick distinct exponents $k_1,\dots,k_r$ all divisible by
\[
  L\ :=\ \mathrm{lcm}_{p\leqslant r}(p-1).
\]
Then for each prime $p\leqslant r$ we have $2^{k_i}\equiv 1\pmod p$, so the shifts $-2^{k_i}$ are all congruent to $-1$ modulo $p$, hence the set $\{-2^{k_1},\dots,-2^{k_r}\}$ is admissible.
Maynard's result therefore yields infinitely many $n$ for which at least $A$ of the numbers $n-2^{k_i}$ are prime.

\subsection*{Gaps between representable odd numbers.} Here comes another closely related problem of Yong-Gao Chen. Let $\mathcal{R}=\{s_1<s_2<\cdots<s_m<\cdots\}$ be the set of all odd numbers which can be represented as the sum of a prime and a power of $2$. Repeatedly, Chen (private communications) asked whether
$$
\limsup_{m\rightarrow\infty}\big(s_{m+1}-s_m\big)=\infty.
$$
This, if true, would be a little unusual in view of Romanoff's theorem asserting that $\PP+\{2^k:k\geqslant 0\}$ has positive lower density (see \cite{Del2020}). 
For a more general conjecture involving the gaps between the elements of $\mathcal{A}_\lambda$, one can refer to the article of Chen and Xu \cite[Conjecture 1.5]{ChenXu2024}.
Cram\'er's conjecture (see e.g., \cite{Cramer1936}) 
$$
\limsup_{n\rightarrow\infty}\frac{p_{n+1}-p_n}{(\log p_n)^2}<\infty
$$
implies that 
$$
\limsup_{m\rightarrow\infty}\frac{s_{m+1}-s_m}{(\log s_m)^2}<\infty.
$$
Indeed, if $s_m=p+2^k$ and $p'$ is the next prime after $p$, then $p'+2^k\in\mathcal{R}$ and $s_{m+1}-s_m\leqslant p'-p$. 
Currently, we cannot even guess whether 
$$
\limsup_{m\rightarrow\infty}\frac{s_{m+1}-s_m}{(\log s_m)^2}=0
$$
or not. 

\section*{Acknowledgments}
We thank Steve Fan for his very detailed explanations on Lemma \ref{lem:prachar}. We also thank the anonymous referee for helpful comments and suggestions that improved the exposition.

\section*{Data availability}
No data was used for the research described in the article.


\begin{thebibliography}{99}
\bibitem{BHP2001}
R.~C.~Baker, G.~Harman, and J.~Pintz, \emph{The difference between consecutive primes.\ II},
Proc.\ Lond.\ Math.\ Soc.\ (3)\ \textbf{83} (2001) 532--562, doi: \nolinkurl{10.1112/plms/83.3.532}.

\bibitem{vanDerCorput1950}
J.~G.~van~der~Corput, \emph{On de Polignac's conjecture},
Simon Stevin \textbf{27} (1950) 99--105.

\bibitem{ChenDing}
Y.-G.~Chen and Y.~Ding,  \emph{On a conjecture of Erd\H{o}s}, C. R. Math. Acad. Sci. Paris  \textbf{360} (2022) 971--974, doi: \nolinkurl{10.5802/crmath.345}.

\bibitem{ChenDing2023}
Y.-G.~Chen and Y.~Ding, \emph{Quantitative results of the Romanov type representation functions}, Q. J. Math. \textbf{74} (2023) no.~4 1331--1359, doi: \nolinkurl{10.1093/qmath/haad022}.

\bibitem{ChenXu2024}
Y.-G.~Chen and J.-Z.~Xu, \emph{On integers of the form $p+2^{k_1^{r_1}}+\cdots+2^{k_t^{r_t}}$}, J. Number Theory \textbf{258} (2024) 66--93, doi: \nolinkurl{10.1016/j.jnt.2023.10.018}.

\bibitem{Cramer1936} H. Cram\'er, \emph{On the order of magnitude of the difference between consecutive prime numbers,} Acta Arith. {\bf2} (1936) 23--46.

\bibitem{Del2020} G. M. Del Corso, I. Del Corso, R. Dvornicich and F. Romani, \emph{On computing the density of integers of the form $2^n+p$,} Math. Comput. {\bf89} (2020) 2365--2386, doi: \nolinkurl{10.1090/mcom/3537}.

\bibitem{DingZhai}
Y.~Ding and W.~Zhai, \emph{A generalization of the Romanoff theorem}, Int. J. Number Theory {\bf22} (2026) no. 1 163--173, doi: \nolinkurl{10.1142/S1793042126500107}.

\bibitem{ElsholtzSchlagePuchta2018}
C.~Elsholtz and J.-C.~Schlage-Puchta, \emph{On Romanov's constant},
Math.\ Z.\ \textbf{288} (2018) 713--724, doi: \nolinkurl{10.1007/s00209-017-1908-x}.

\bibitem{Erdos1950}
P.~Erd\H{o}s, \emph{On integers of the form $2^k+p$ and some related problems}, Summa Bras. Math. \textbf{II} (1950) 113--123.

\bibitem{GuthMaynard2024}
L.~Guth and J.~Maynard, \emph{New large value estimates for Dirichlet polynomials}, arXiv:2405.20552v2 (2026).

\bibitem{HalberstamRichert}
H.~Halberstam and H.-E.~Richert, \emph{Sieve Methods}, Academic Press, 1974.

\bibitem{IwaniecKowalski}
H.~Iwaniec and E.~Kowalski, \emph{Analytic Number Theory},
American Mathematical Society Colloquium Publications, vol.~53, American Mathematical Society, 2004.


\bibitem{LiShortIntervals}
R.~Li, \emph{The number of primes in short intervals and numerical calculations for Harman's sieve}, preprint, arXiv:2308.04458.

\bibitem{Maynard2015}
J.~Maynard, \emph{Small gaps between primes}, Ann.\ of Math.\ (2) \textbf{181} (2015) no.~1 383--413, doi: \nolinkurl{10.4007/annals.2015.181.1.7}.

\bibitem{dePolignac1849a}
A.~de~Polignac, \emph{Six propositions arithmologiques d\'eduites du crible d'\'Eratosth\`ene},
Nouvelles annales de math\'ematiques \textbf{8} (1849) 423--429.

\bibitem{dePolignac1849b}
A.~de~Polignac, \emph{Recherches nouvelles sur les nombres premiers},
C.\ R.\ Acad.\ Sci.\ Paris \textbf{29} (1849) 738--739.

\bibitem{Romanoff1934}
N.~P.~Romanoff, \emph{\"Uber einige S\"atze der additiven Zahlentheorie},
Math.\ Ann.\ \textbf{109} (1934) 668--678, doi: \nolinkurl{10.1007/BF01449161}.

\end{thebibliography}
\end{document}